\documentclass{amsart}
\usepackage{dino}
\usepackage{thm-restate}
\makeatletter
\newcommand{\rep}{\mathbf{Rep}}
\newcommand{\norep}{\mathbf{NoRep}}
\DeclareMathOperator{\esse}{\mathcal{S}}
\newcommand{\baire}{\omega^{<\omega}}
\usepackage{comment}
\newcommand{\ex}{\mathbf{Ex}}
\newcommand{\Baire}{\omega^\omega}
\newcommand{\Cantor}{2^{\omega}}

\newcommand{\thpi}[2]{\mathrm{Th}_{\Pi_{#1}^{\mathrm{in}}}(#2)}
\newcommand{\str}[1]{\langle #1 \rangle}
\renewcommand{\bc}{\mathbf{Bc}}
\newcommand{\WO}{\mathbf{WO}}
\newcommand{\NWO}{\mathbf{NWO}}
\newcommand{\WF}{\mathbf{WF}}
\newcommand{\IF}{\mathbf{IF}}

\begin{document}
\title{Uniformity in learning structures}
\author[Cipriani]{Vittorio Cipriani}
\author[Rossegger]{Dino Rossegger}

\thanks{The first author was supported by the Austrian Science Fund (FWF) 10.55776/P36781.
The last author was supported by the Austrian Science Fund (FWF) 10.55776/PIN1878224.}

\address[Cipriani,Rossegger]{Institute of Discrete Mathematics and Geometry, Technische Universit{\"a}t Wien, Wiedner Hauptstraße 8-10/104, 1040 Wien, Austria}
\email{vittorio.cipriani17@gmail.com} 
\urladdr{\url{vittoriocipriani.github.io}}
\email{dino.rossegger@tuwien.ac.at}
\urladdr{\url{drossegger.github.io}}

\begin{abstract}
The standard framework for studying learning problems on algebraic structures assumes that the structures in the target family are pairwise nonisomorphic. Under this assumption, the most widely investigated learning criterion—$\ex$-learning—becomes inherently equivalent to the well-known paradigm of $\bc$-learning. This paper explores what happens when the nonisomorphism requirement is removed and analyzes the extent to which these two learning criteria remain uniformly equivalent.
\end{abstract}
\maketitle


\section{Introduction}
Algorithmic learning theory originates in the work of Gold \cite{Gold67} and Putnam \cite{putnam1965trial} from the 1960s, and encompasses several formal frameworks for algorithmic learning theory. Broadly speaking, this research program models how a learner acquires systematic knowledge about an environment as increasingly large amounts of data become available. Classical paradigms focused mainly on inferring formal languages or computable functions (see, e.g., \cite{lange2008learning,zz-tcs-08}).

More recently, researchers have begun studying the complexity of the isomorphism relation on countable structures in this context. The standard framework draws on concepts and methods from computable structure theory, and was introduced and later refined in \cite{bazhenov2020learning,bazhenov2021learning}. In this setting, a learner $\mathbf{L}$ is given a family of structures $\vec{\A}$ in a fixed vocabulary. A structure $\mathcal{S}$, isomorphic to some $\A_i \in \vec{\A}$, is then revealed over infinitely many stages, each stage providing a finite fragment of $\mathcal{S}$. At every stage $s$, the learner outputs a hypothesis $i_s \in \omega$ indicating which $\A_i$ it believes $\mathcal{S}$ is isomorphic to. Several notions of success have been considered, with $\ex$-learning being the most prominent: a family $\vec{\A}$ is $\ex$-learnable if there exists a learner which eventually stabilizes on the correct conjecture.

Another important criterion in the learning of formal languages and computable functions is $\bc$-learnability: a family is $\bc$-learnable if there is a learner that, from some stage onward, always outputs a correct index for the target language or function. This criterion is meaningful because a given formal language or computable function has many distinct indices.

When learning isomorphism on structures one usually assumes that the structures in $\vec \A$ are pairwise nonisomorphic. This assumption is justified by the fact that $\vec{\A}$ is countable: if one formalizes $\mathbf{L}$ as a continuous map between suitable spaces, then (so the argument goes) one may harmlessly add countably much auxiliary information without affecting continuity. However, one might criticize that the process of going from arbitrary $\vec{\A}$ to $\vec{\A}'$ containing exactly one element of every isomorphism class occurring in $\vec{\A}$ must be inherently non-uniform, as the isomorphism relation on countable structures is very complicated. If a vocabulary $\tau$ contains one binary relation symbol, then the isomorphism relation on $\tau$-structures is not Borel, and thus figuring out whether two elements in a sequence are isomorphic cannot be done uniformly using any countable procedure. One might thus expect that there is no Borel function $f$ that given as input a sequence $\vec\A$ that is $\ex$-learnable, outputs a sequence $f(\vec \A)$ of pairwise non-isomorphic structures that is $\ex$-learnable.

Furthermore, assuming pairwise non-isomorphism, $\ex$- and $\bc$-learning become trivially equivalent. However, this raises the following natural question: what happens if we drop the requirement that the structures in $\vec{\A}$ are pairwise nonisomorphic?

As \Cref{thm:bfchar} shows, even without this assumption, a family $\vec{\A}$ is $\ex$-learnable if and only if it is $\bc$-learnable. On the other hand, \cref{thm:nofunction} shows that this equivalence is again inherently non-uniform, by showing that there is no function that transforms a $\bc$-learner into an $\ex$-learner. Only if we additionally give the family to be learned as input, do we obtain a Borel function that transforms $\bc$-learners for that family into $\ex$-learners. Thus, one might say that the equivalence is uniform in the family.
At last, we confirm in \cref{thm:repnorep} our suspicion that there is no Borel function reducing the set of learnable sequences to the set of learnable pairwise non-isomorphic sequences.

  \section{Background}
  \label{sec:background}
  \subsection{Descriptive set theory}
  We will use standard techniques and notions from descriptive set theory and point the reader to~\cite{kechris2012} for a detailed treatment of the subject. We consider countable structures in relational vocabulary $\tau=(R_i/a_i)_{i\in I}$ and universe $\omega$ as elements of the compact Polish space
  \[
    Mod(\tau)\subseteq \prod_{i\in I} ( 2^{\omega} )^{a_i}
  \]
  whose topology is the usual product topology.
Many familiar classes of structures such as the class of linear orders $\mathrm{LO}$ are closed subsets of $Mod(\tau)$ and thus form a Polish space via the subspace topology.

One space that is central to our analysis is the \define{space $C(Mod(\tau),\omega^\omega)$ of continuous functions} from $Mod(\tau)\to \omega^\omega$ with the topology induced by the uniform metric. Since $Mod(\tau)$ is compact Polish and $\omega^\omega$ is Polish, so is $C(Mod(\tau),\omega^\omega)$ (\cite[Theorem 4.19]{kechris2012}).
  \subsection{Algorithmic learning}
  Algorithmic learning of the isomorphism relation of countable structures was first introduced in \cite{FKS-19} and can be informally described via the following game. A learner $\mathbf{L}$ is presented with a family of structures $\vec \A=(\A_i)_{i \in \omega}\in Mod(\tau)^\omega$ and, at each stage $s$, with a finite substructure $\S\restrict s$ having domain $\{0,\dots,s\}$ of a structure $\S\in Mod(\tau)$. They are only guaranteed that $\bigcup_{s\in\omega} \S\restrict s=\S$ and have to guess at each stage $s$ to which $\A_i$ the structure $\S$ is isomorphic. Note that, in this paper, we allow the possibility that $\S$ is not isomorphic to any member of $\vec \A$, as is done, for example, in \cite{FKS-19}.

 Algorithmic learning theory considers several criteria for determining when a learner succeeds. The most widely studied in the context of structures is \define{explanatory learning} ($\ex$-learning): a learner successfully learns $\vec{\A}$ if, when presented with some $\S \cong \A_j$, the limit $i = \lim_s i_s$ of its stagewise guesses exists and satisfies $\A_i \cong \A_j$. Another well-established criterion—prominently investigated in the learning of formal languages and computable functions—is \define{behaviorally correct learning} ($\bc$-learning). Its definition readily extends to structures by dropping the requirement that the guesses converge; instead, we merely require that $\A_{i_s} \not\cong \S$ for only finitely many stages $s$. We now define these notions formally.

  \begin{definition}\label{def:learning}
    Let $K\subseteq Mod(\tau)$ be isomorphism invariant and $\vec \A=(\A_i)_{i\in\omega}\in K^\omega$. 
    \begin{enumerate}
      \item A \define{learner} $\mathbf{L}$ for $\vec \A$ is a continuous function $K\to\omega^\omega$.\footnote{Here continuous means continuous in the subspace topology on $K$.}
      \item The family $\vec \A$ is \define{explanatory learnable}, short $\ex$-learnable, if there exists a learner $\mathbf{L}$ such that for every $i$ and $\S\cong \A_i$, $\lim_n \mathbf{L}(\S)(n)\downarrow$ and $\A_{\lim_n \mathbf{L}(\S)(n)}\cong \A_i$. 
      \item The family $\vec \A$ is \define{behaviorally correct learnable}, short $\bc$-learnable, if there exists a learner $\mathbf{L}$ such that for every $i$ and $\S\cong \A_i$, $\forall^{\infty} n\  \A_{\mathbf{L}(\S)(n)} \cong\A_i$.
    \end{enumerate}
%
%
		\end{definition}
    \begin{remark}
      In the literature $\ex$-learnability is usually only defined for families $\vec\A$ whose elements are pairwise non-isomorphic. Our definition is more general but all results from the literature for $\ex$-learnability hold in a relativized version for our definition.
    \end{remark}
    \begin{remark}
  The definitions of learners in~\cite{bazhenov2020learning,bazhenov2021learning,rossegger2024} are different from ours but our definition has the advantage that learners are elements of $C(K,\omega^\omega)$. It is not difficult to see that the definitions give rise to equivalent notions of learnability. In~\cite{bazhenov2020learning,bazhenov2021learning} learners are assumed to be functions from finite $\tau$-structures to $\omega$. To see that our definition is equivalent, suppose that $f$ is such a function and define the learner $\mathbf{L}(\S)(n)=f(\S\restrict k)$ where $k$ is maximal less or equal to $n$ such that $f(\S\restrict k)$ is defined or $0$, if $f(\S\restrict k)$ is undefined for every $k\leq n$. Then clearly $\lim_n \mathbf{L}(\S)(n)\downarrow=\lim_n f(\S\restrict n)$. On the other hand, to define $f$ given $\mathbf{L}$ simply let $f(\S\restrict n)=\mathbf{L}(\S)(k)$ where $k$ is largest such that $use(\mathbf{L}(\S)(k))\subseteq S\restrict n$.

    In~\cite{rossegger2024} a learner was defined to be a sequence of continuous functions $f_i:K\to \omega$ and explanatory learning was defined by considering $\lim_{i} f_i$. Again, it is not hard to see that one can computably transform a learner given by our definition to a learner in the sequence definition and vice versa.
  \end{remark}

    Bazhenov, Fokina and San Mauro~\cite{bazhenov2020learning} gave the following syntactic characterization of $\ex$-learnability. We will extend this theorem in \Cref{thm:bfchar}.
\begin{theorem}[{\cite[Theorem 3.1]{bazhenov2020learning}}]
\label{theorem:syntacticcharacterization}
For any family of pairwise non-isomorphic structures $\vec\A$ the following are equivalent.
\begin{enumerate}
    \item $\vec \A$ is $\ex$-learnable;
    \item $\vec\A$ has $\Sinf{2}$ quasi-Scott sentences. That is, $\Sinf{2}$ sentences $(\varphi_i)_{i \in \omega}$ 
      such that $\A_i \models \varphi_j$ if and only if $\A_i\cong\A_j$.
\end{enumerate}
   
\end{theorem}


  %
 

\section{When $\bc$  translates to $\ex$}
\label{sec:whentranslates}
The following theorem tells us that we have no hope of translating a $\bc$-learner into an $\ex$-learner without knowing the target family we want to learn.

\begin{theorem}
\label{thm:nofunction}
Suppose that there is $K\subseteq Mod(\tau)$ with $\A_1,\A_2\in K$ such that $\A_1\not\equiv_2\A_2$. Then, there is no function $f:C(K,\Baire) \rightarrow C(K,\Baire)$ such that for every $\vec \A\in K^\nat$, if $\mathbf{L}$ is a $\bc$-learner for $\vec \A$ then $f(\mathbf{L})$ is an $\ex$-learner for $\vec \A$.
\end{theorem}
\begin{proof}
Assume that such an $f$ exists. 
Let $\A_1,\A_2$ be two $\tau$-structures such that $\A_1 \not\equiv_2 \A_2$ and consider the family of structures $\vec \A$ where $\vec \A (2i)=\A_1$ and $\vec \A (2i+1)=\A_2$. Without loss of generality, let $\varphi_1=\exists \bar x \forall \bar y \psi(\bar x, \bar y)$ and $\varphi_2=\exists \bar x \forall \bar y \theta(\bar x, \bar y)$ be the $\Sinf{2}$-quasi Scott sentences for $\A_1$ and $\A_2$ respectively where $\psi$ and $\theta$ are quantifier-free formulas.

We order tuples via the ordering given by some standard bijection between tuples and natural numbers.
Given a structure $\esse$ isomorphic either 
to $\A_1$ or $\A_2$, we define $a_s := \min \{\bar x : (\forall \bar y \leq s) \esse\restrict s \models \psi(\bar x,\bar y) \}$ and $b_s := \min \{\bar x : (\forall \bar y \leq s) \esse\restrict s \models \theta(\bar x,\bar y)\}$, and we let $c_{a}(s) := |\{t \leq s : a_s = a_t\}|$ and $c_{b}(s) := |\{t \leq s : b_s = b_t\}|$. We define a learner $\mathbf{L}$ for $\vec \A$ as follows:
\[\mathbf{L}(\esse\restrict s)=\begin{cases}
    2s+2 & \text{if } c_{a}(s)> c_{b}(s)\\
    2s+1 & \text{if } c_{a}(s)\leq c_{b}(s)
\end{cases}\]
Clearly, $\mathbf{L}$ is a $\bc$-learner for $\vec\A$. By our assumption, $f(\mathbf{L})$ is an $\ex$-learner for $\vec\A$ and hence given a structure $\esse \cong \A_1$, $f(\mathbf{L})(\esse)$ will eventually converge to some $N \in \omega$.

Consider the family of structures $\vec\A'$ defined as $\vec \A$ except for the fact that, 
\begin{itemize}
    \item if $N=2k$ for some $k \in \omega$, $\vec \A'(N)=\A_2$;
    \item if $N=2k+1$ for some $k \in \omega$, $\vec \A'(N)=\A_1$.
    \end{itemize}
Clearly $\mathbf{L}$ is still a $\bc$-learner for $\vec\A'$ but we have that $f(\mathbf{L})(\esse)=N$, i.e., $f(\mathbf{L})$ fails to $\ex$-learn $\vec \A'$.
\end{proof}

\Cref{proposition:exequalsbc} will show that such a translation can be pursued if we assume that the family to be learned is given as input to the translating function.

\subsection{Back-and-forth relations and learnability}
We will recast \cref{theorem:syntacticcharacterization} in terms of the asymmetric version of Karp's back-and-forth relations. Asymmetric back-and-forth relations first appeared in~\cite{barwise1973} and their theory was systematically developed by Ash and Knight, see~\cite{ash2000} for details. For any two $\tau$-structures $\A$ and $\B$ and $\alpha<\omega_1$, the \define{$\leq_\alpha$ back-and-forth relations} are inductively defined as follows.
\begin{align*}
(\A,\ba)\leq_0(\B,\bb) \iff& \text{for the first $|\ba|$ formulas $(\phi_i)_{i<|\ba|}$ in an enumeration}\\&\text{of the atomic $\tau$-formulas}\ \A\models \phi_i(\ba) \iff \B\models \phi_i(\bb)\\
  (\A,\ba)\leq_\alpha (\B,\bb) \iff& (\forall \beta<\alpha)\forall \bd \exists \bar c\  (\B,\bb\bd)\leq_\beta (\A,\ba\bar c)
\end{align*}
We write $\A\leq_\alpha \B$ for $(\A,\emptyset)\leq_\alpha (\B,\emptyset)$ and $(\A,\ba)\equiv_\alpha(\B,\bb)$ if $(\A,\ba)\leq_\alpha (\B,\bb)$ and $(\B,\bb)\leq_\alpha (\A,\ba)$.
The single most important fact about asymmetric back-and-forth relations is the following.
\begin{theorem}[Karp~\cite{karp1965b}]
\label{theorem:karp}
For any two $\tau$-structures $\A$ and $\B$ and $\alpha<\omega_1$, $\A\leq_\alpha \B$ if and only if ${\Pinf{\alpha}}\text{-th}((\A,\ba))\subseteq {\Pinf{\alpha}}\text{-th}((\B,\bb))$.
\end{theorem}

\begin{theorem}\label{thm:bfchar}
  For any family of structures $\vec\A$ the following are equivalent.
  \begin{enumerate}
    \item\label{it:bfchar_ex} $\vec \A$ is $\ex$-learnable
    \item\label{it:bfchar_bc} $\vec \A$ is $\bc$-learnable
    \item\label{it:bfchar_bf} $\forall i \exists \ba \forall j \forall\bb (\A_i,\ba)\not\leq_1 (\A_j,\bb)\lor \A_i\cong \A_j$
    \item\label{it:bfchar_s2} $\vec \A$ has $\Sinf{2}$ quasi-Scott sentences.
  \end{enumerate}
  Furthermore, if the structures in $\vec \A$ are pairwise non-isomorphic, then (3) can be replaced by
  \begin{enumerate}
    \item[(3a)]\label{it:bfcar_bfnorep} $\forall i \exists \ba \forall j\neq i\forall\bb (\A_i,\ba)\not\leq_1 (\A_j,\bb)$.
  \end{enumerate}
\end{theorem}
\begin{proof}
  The fact that \cref{it:bfchar_ex} implies \cref{it:bfchar_bc} is trivial. To see that the converse implication holds, given a $\bc$-learner $\mathbf{L}$ for a family of structures $\vec \A$,  
  we define an $\ex$-learner $\mathbf{L}'$ for $\vec \A$ as follows. Given a structure $\esse$, let $\mathbf{L}'(\esse)(n):= \mu i [\A_i \cong \mathbf{L}(\esse)(n)]$. It is clear that $\mathbf{L}$ $\ex$-learns $\vec \A$.
  
  The fact that \cref{it:bfchar_s2} can be replaced by Item 3a if the structures in $\vec\A$ are pairwise non-isomorphic is immediate.

 We now prove that \cref{it:bfchar_ex} implies \cref{it:bfchar_s2}. Let $\vec\A$ be an $\ex$-learnable family of structures. Let $\vec\A'=(\A'_i)_{i \in I}$ for $I$ being either $\omega$ or an initial segment of $\omega$ be the family of structures defined by $\A_{i}'=\A_{m}$ where $m:=\mu j[\A_j \not\cong \A_k' \text{ for every }k < i]$. Clearly $\vec\A'$ is a family of pairwise non-isomorphic structures and hence, by \cref{theorem:syntacticcharacterization} it has $\Sinf{2}$ quasi-Scott sentences $(\varphi_i)_{i \in I}$. For any $i$, let $\psi_i:=\varphi_{k}$ if and only if $\A_i \cong \A_k'$. By definition of $\vec\A'$ for every $i$ there is a unique such $k$. It is immediate that $(\psi_i)_{i \in \omega}$ are the desired $\Sinf{2}$ quasi-Scott sentences for $\vec \A$.

 The proof that \cref{it:bfchar_s2} implies \cref{it:bfchar_ex} is a standard argument. We will sketch it here. Let $(\varphi_i)_{i \in \omega}$ be $\Sinf{2}$ quasi-Scott sentences for $\vec \A$ where without loss of generality $\varphi_i=\exists \bar x \forall \bar y \theta_i(\bar x,\bar y)$. Define $\mathbf{L}(\S)(n)=\mu i [ \S\restrict n \models \varphi_i]$. Then it is not hard to see that if $\S\cong \A_i\in \vec \A$, then $\lim\mathbf{L}(\S)$ is the least $j$ such that $\S\cong \A_j\in \vec \A$.

 To prove that \cref{it:bfchar_s2} implies \cref{it:bfchar_bf}, let $(\varphi_i)_{i \in \omega}$ be $\Sinf{2}$ quasi-Scott sentences for $\vec\A$. Without loss of generality $\varphi_i$ is of the form $\exists \overline{x} \psi_i$ for $\psi_i$ a $\Pinf{1}$-formula.
 Since $\varphi_i$ is a quasi-Scott sentence, there is some $\overline{a}$ such that $(\A_i,\overline{a}) \models \psi_i$. We claim that for every $j \neq i$ and for every $\overline{b}$, either $\A_i \cong \A_j$ or $\thpi{1}{(\A_i,\overline{a})}\not\subseteq \thpi{1}{(\A_j, \overline{b})}$, which, by \cref{theorem:karp}, is equivalent to $(\A_i,\overline{a}) \nleq_1 (\A_j,\overline{b})$. So suppose that $\A_i \not\cong \A_j$, then by $\phi_i$ being a quasi-Scott sentence for $\A_i$, for every $j$ and $\overline{b}$, $(\A_j,\overline{b})\not\models \psi_i$. Hence, $\thpi{1}{(\A_i,\overline{a})}\not\subseteq \thpi{1}{(\A_j, \overline{b})}$ and so $(\A_i,\bar a)\not\leq_1(\A_j,\bar b)$ for every $\bar b$ and $\A_j\not\cong \A_i$ as required.

 To conclude the proof of the theorem we now prove that \cref{it:bfchar_bf} implies \cref{it:bfchar_s2}. Let $f: \omega \rightarrow \baire$ be the Skolem function that given $i$ outputs the corresponding $\overline{a}$ witnessing that for every $j$ and for every $\overline{b}$, $(\A_i,\overline{a}) \nleq_1 (\A_j,\overline{b})$ or $\A_i \cong \A_j$. So let $\psi_i$ be the conjunction of all finite universal formulas that are true of $(\A_i,f(i))$ and let $\phi_i := \exists \overline{x} \psi_i(\overline{x})$. Since $(\A_i,f(i)) \models \psi_i(f(i))$, we have that $\A_i \models \phi_i$. We claim that $(\phi_i)_{i \in \omega}$ are the desired $\Sinf{2}$ quasi-Scott sentences. To prove this fix $i$, and given $\A_j\not\cong \A_i$ suppose that $\A_j \models \varphi_i$. By definition of $\varphi_i$, this means that there is a $\overline{b}$ such that $(\A_j,\overline{b}) \models \psi_i(\overline{b})$. Since $\psi_i$ was defined to be the conjunction of all universal formulas that are true in $(\A_i,f(i))$, this means that $\thpi{1}{\A_i}\subseteq \thpi{1}{\A_j}$ and, by \cref{theorem:karp}, we have that $(\A_i,f(i)) \leq_1 (\A_j,\overline{b})$, contradicting our assumption.
\end{proof}

The proof showing that \cref{it:bfchar_bc} implies \cref{it:bfchar_ex,} in \Cref{thm:bfchar} can be adapted to prove the following proposition.


 \begin{proposition}
  \label{proposition:exequalsbc}
  Let $K\subseteq Mod(\tau)$ be isomorphism invariant. 
  There is a Borel function $f: K^\omega \times C(K,\Baire) \rightarrow K^\omega \times C(K,\Baire)$ such that if $\mathbf L$ is a  $\bc$-learner for $\vec \A\in K^\omega$ then $f(\vec\A,\mathbf L)$ is an $\ex$-learner for $\vec \A$.
  
  \end{proposition}
  \begin{proof}
  
    Suppose that $\vec\A$ is $\bc$-learnable by the learner $\mathbf L$. Define $f$ by 
    \[f(\vec \A,\mathbf{L})(\S)(n):=\mu i [ \vec\A_i\equiv_2 \mathbf{L}(\S)(n)].\]
    
    As $\equiv_2$ is $\Pi^0_4$, $f$ is Borel and clearly if $\mathbf L$  $\bc$-learns $\vec \A$, then $f(\vec \A,\mathbf{L})$  $\ex$-learns $\vec \A$. Indeed, if $\vec \A_{\mathbf{L}(\S)(n)}\cong\vec \A_{\mathbf{L}(\S)(m)}$, then $f(\vec \A,\mathbf{L})(\S)(n)=f(\vec \A,\mathbf{L})(\S)(m)$.
  \end{proof}
  The proof of this theorem can be adapted to give some insight into a question arising in~\cite{rossegger2024}. There the learnability of arbitrary equivalence relations on Polish spaces was studied. Taking into account the proposed changes remarked after \cref{def:learning}, we say that an equivalence relation $E$ on a Polish space is \define{uniformly $\ex$-learnable} if there is a function $\mathbf{L}\in C(X^\nat\times X,\Baire)$ such that for all $\vec x\in X^\nat$ and $x\in X$ if there is $m$ so that $\vec x_m\mathrel{E} x$, then $\lim_n \mathbf{L}(\vec x,x)\downarrow=l$ and $\vec x_l\mathrel{E} x$. Similarly, an equivalence relation $E$ is \define{uniformly $\bc$-learnable} if there is a function $\mathbf{L}\in C(X^\nat\times X,\Baire)$ such that for all $\vec x\in X^\nat$ and $x\in X$ if there is $m$ so that $\vec x_m\mathrel{E} x$, then for cofinitely many $l$ $x_{\mathbf{L}(\vec x, x)(l)}\mathrel{E} x$. It was shown in~\cite{rossegger2024} that an equivalence relation is uniformly $\bc$-learnable if and only if it is uniformly $\ex$-learnable. While it is clear that uniform $\ex$-learnability implies uniform $\bc$-learnability, the other implication is quite involved. In the proof one first obtains a $\bSigma^0_2$-definition that works for $E$ modulo countably many $E$-classes. By non-uniformly fixing the definitions for these $E$-classes, one obtains a Borel definition, which can then be shown to be $\bSigma^0_2$ after all. The latter is enough to guarantee $\ex$-learnability of $E$. As one can hopefully see from this sketch, the proof is not very direct and one might wonder whether there is a direct proof which skips the non-uniform part of the construction. We can show that in the case where $E$ is isomorphism, this can be achieved.
  \begin{proposition}\label{prop:uniformexequalsbc}
    Let $K\subseteq Mod(\tau)$ be isomorphism invariant. Then there is a Borel function $f: C(K^\nat\times K,\Baire)\to C(K^\nat\times K,\Baire)$ such that if $\mathbf{L}$ uniformly $\bc$-learns $\cong\restrict K$, then $f(\mathbf{L})$ uniformly $\ex$-learns $\cong\restrict K$. 
  \end{proposition}
  \begin{proof}
    Suppose that $\mathbf{L}$ uniformly $\bc$-learns $\cong\restrict K$. Define $f$ by
    \[ f(\mathbf{L})(\vec \A,\S)(n):= \mu i [ \vec \A_i\equiv_2\mathbf{L}(\S)(n)].\]
    By an argument identical to the one in \cref{proposition:exequalsbc}, $f$ is as required.
  \end{proof}
  The proofs of \cref{proposition:exequalsbc,prop:uniformexequalsbc} rely on the presence of invariants in the form of back-and-forth relations. Back-and-forth relations can be formulated for any equivalence relation induced by the action of a Polish group~\cite{hjorth2000}. Thus, with a little more effort, one could probably replicate the proof of \cref{prop:uniformexequalsbc} for such equivalence relations, streamlining the argument from~\cite{rossegger2024} in the orbit setting.

\section{Learnability vs complexity of families of structures}
\label{sec:complexity}
\Cref{proposition:exequalsbc} shows that, assuming we know the family of structures to be learned, $\ex$- and $\bc$-learning coincide. The next section shows that learnability has, in some sense, little to do with the complexity of the families to be learned. In other words, we retrieve the fact that $\bc$-learning a family is a more difficult task.
\subsection{Trees and analytic sets}
   Given $\sigma,\tau \in \omega^{n}$, we define 
   \[\sigma*\tau:=\str{ \sigma(0),\tau(0),\dots ,\sigma(n-1),\tau(n-1) }.\] We can extend this definition to infinite sequences $f,g\in\omega^\omega$ by setting $f*g(n)=\sigma*\tau(n)$ where $\sigma\prec f$ and $\tau\prec g$ with $|\sigma|=|\tau|=n$. Given trees $T$ and $S$, the \define{interleaving of $T$ and $S$} is 
   \[
       T*S:=\{ \sigma*\tau: |{\sigma}|=|{\tau}| \land \sigma \in T \land \tau \in S\}.
   \]
   Notice that $T*S$ is a tree and is well-founded if and only if both $T$ and $S$ are well-founded. 

   Recall that the \define{Kleene-Brouwer order} is the order on $\omega^{<\omega}$ defined by
   \begin{align*}
       \sigma\leq_{KB}\tau := \sigma\preceq \tau\ &\text{or
   $\sigma$ and $\tau$ are incomparable}\\
                                  &\text{and for the least $i$ such that $\sigma(i)\neq \tau(i)$, $\sigma(i)<\tau(i)$.}
   \end{align*}
   The most important property of $\leq_{KB}$ is that if $T\subseteq \omega^\omega$ is a tree, then the Kleene-Brouwer order restricted to $T$, $KB(T)$, is a well-order if and only if $T$ is well-founded. Ill-founded trees and linear orders are the archetypical $\Sigmab^1_1$ complete sets.
   \begin{theorem}[folklore]
   \label{theorem:folklore}
       The sets $\WO=\{ L\in \mathrm{LO}: L \text{ is a well-order}\}$ and $\WF=\{T\in \mathrm{Tr}: T\text{ is well-founded}\}$ are $\Pib^1_1$-complete. Likewise, $\NWO=\mathrm{LO}\setminus \WO$ and $\IF=\mathrm{Tr}\setminus\WF$ are $\Sigmab^1_1$-complete.
   \end{theorem}

\subsection{Pseudo well-orders and their properties}
Harrison~\cite{harrison1968} showed the existence of so-called pseudo well-orders; computable linear orders that are not well-ordered but do not contain hyperarithmetic descending sequences. As most results in computability theory, his result relativizes to arbitrary oracles $x\in 2^\omega$. We thus say that a linear order $L$ is a \define{pseudo well-order for $x\in 2^\omega$} if $x\geq_T L$, $L$ is not well-ordered and there are no $\Delta^1_1(x)$ descending chains in $L$. We say that $L$ is a \define{pseudo well-order} if it is a pseudo well-order for some $x\in 2^\omega$. Below we state the relativized versions of some structure-theoretic facts about pseudo well-orders.

\begin{lemma}[cf. {\cite[Lemma VI.7]{montalban2023a}}]
  \label{lemma:initialsegment}
   If $L$ is a pseudo well-order for $x\in 2^\omega$, then $L$ has order-type $\omega_1^{x} + \omega_1^{x} \cdot \mathbb{Q}+\alpha$ where $\alpha<\omega_1^x$.
\end{lemma}
\begin{lemma}[cf. {\cite[Lemma VI.11]{montalban2023a}}]
\label{lemma:montalban1}
   There is a computable operator $H$ such that for every $x \in 2^{\omega}$, $H^x$ is a pseudo well-order for $x$. Furthermore, $H^x$ has order-type $\omega_1^x +\omega_1^x \cdot \mathbb Q$.
\end{lemma}
    
The following fact about pseudo well-orders is well known; for example it appears in the usual proof that shows that the isomorphism relation on linear orderings is $\pmb \Sigma_1^1$-complete (cf.~\cite[Lemma XI.2]{montalban2023a}).
\begin{lemma}
\label{lemma:illfoundedlinearorder}
      Suppose that $L$ is a pseudo well-order, then there is $x\in 2^\omega$ such that 
   \[ L\cdot\omega\cong \omega_1^x + (\omega_1^x \cdot \mathbb{Q})\]
\end{lemma}
\begin{proof}
  By \cref{lemma:initialsegment} we may assume that $L$ has order-type $\omega_1^x+\omega_1^x\cdot \mathbb Q +\alpha$ where $\alpha<\omega_1^x$. We then get that
  \begin{align*}
    (\omega_1^x+\omega_1^x\cdot \mathbb Q +\alpha)\cdot \omega&=\omega_1^x+\omega_1^x\cdot \mathbb Q +\alpha+\omega_1^x+\omega_1^x\cdot \mathbb Q +\alpha+\omega_1^x+\omega_1^x\cdot \mathbb Q +\alpha+\cdots\\
                                         &=\omega_1^x+\omega_1^x\cdot \mathbb Q
  \end{align*}
  where the last equation holds because $\alpha+\omega_1^x=\omega_1^x$ by ordinal arithmetic and $L\cdot \mathbb Q\cong L\cdot \mathbb Q+L+L\cdot \mathbb Q$ for any countable linear ordering $L$.
\end{proof}

\begin{lemma}
\label{lemma:wellorderprop}
    Suppose that $L \in \WO$, then for every $x \in \Cantor$
    \begin{enumerate}\tightlist
        \item\label{it:wellorderprop_cong} $ L\cdot \omega \not\cong\ \omega_1^x+ (\omega_1^x\cdot \mathbb{Q})$,
        \item\label{it:wellorderprop_bf} $ L\cdot \omega \equiv_2\ \omega_1^x+ (\omega_1^x\cdot \mathbb{Q})$.
    \end{enumerate}
\end{lemma}
\begin{proof}
    The proof of \cref{it:wellorderprop_cong} is immediate as $L\cdot \omega \in \WO$ and  $\omega_1^x+ (\omega_1^x\cdot \mathbb{Q}) \not\in \WO$.

  To prove \cref{it:wellorderprop_bf}, we prove the two inequalities. To see that $\omega_1^x+ (\omega_1^x\cdot \mathbb{Q}) \leq_2 L\cdot \omega$ we define a winning strategy for the $\exists$-player in the corresponding back-and-forth game played in two rounds. Without loss of generality the tuple picked by the $\forall$-player in $L\cdot \omega$ partitions this ordering into $n$ non-empty intervals $I_1,\dots ,I_n$. In order for the $\exists$-player to win it is sufficient to find a tuple partitioning $\omega_1^x+(\omega_1^x \cdot \mathbb{Q})$ into $n$ intervals $J_1,\dots J_n$ such that $I_i\leq_1 J_i$. Using the fact that for any two linear orders $L,K$, $K\leq_1 L$ if and only if $|K| \geq |L|$, one easily sees that the $\exists$-player succeeds in doing so by picking limit elements in the first $\omega_1^x$ if the corresponding left interval is infinite and suitable successor elements if the left interval is finite. The last interval is infinite in both orderings. 
    
To see that $L\cdot \omega \leq_2\ \omega_1^x+ (\omega_1^x\cdot \mathbb{Q})$, we again define a winning strategy for the $\exists$-player in the corresponding back-and-forth game. We assume without loss of generality that the $\forall$-player picks an ordered tuple that partitions $\omega_1^x + (\omega_1^x \cdot \mathbb{Q})$ into $n$ non-empty intervals $I_1,\dots, I_{n}$ in the first round. At this point, for the $\exists$-player to win, it suffices to take a partition $J_1,\dots,J_{n}$ of $L\cdot \omega$ such that for every $i \leq n$, $I_i \leq_1 J_i$. They pick the first $n-1$ elements of $L\cdot\omega$ as $J_1,\dots,J_{n-1}$ and all elements larger than the first $n-1$ elements as $J_{n}$. As the partition $I_{n}$ is necessarily infinite and all other partiations $J_i$ are of size $1$, $|I_{i}|\geq |J_{i}|$ for all $i$, as required.\end{proof}
    

We define two subsets of $Mod(\tau)^\omega$ connected to learning.
    \begin{align*}
      \rep&:=\{\vec \A : \vec \A \text{ is $\ex$-learnable}\} \\
    \norep&:=\rep\cap \{ \vec\A: \forall i,j \A_i\not\cong\A_j\}\end{align*}
The set $\rep$ is just the set of learnable families, while $\norep$ has the additional requirement that elements of these families need to be pairwise non-isomorphic. We will now analyze the descriptive complexity of these sets.
\begin{lemma}
  \label{lem:bcsigma11complete}
  The set $\rep$ is $\pmb\Sigma_1^1$-complete.
  \end{lemma}
  \begin{proof}
    By \cref{it:bfchar_bf} of \cref{thm:bfchar} we have that $\rep$ is $\Sigma_1^1$. Hence, to prove the theorem we need to prove that for any $\bSigma_1^1$ set $A$, there is a continuous reduction  from $A$ to $\rep$. Our proof heavily relies on the proof of \cite[Lemma XI.2]{montalban2023a} showing that the isomorphism relation on $\LO$ is $\Sigma_1^1$-complete.
  
  Without loss of generality, assume that $A\subseteq 2^\omega$ is $\Sigma^1_1(y)$ for some fixed $y\in 2^\omega$. Our reduction will  map $x\in 2^\omega$ to a family $f(x)=(\A_i)_{i \in \omega}$ where
  \begin{itemize}
    \item $\A_{2i} \cong \omega_1^{x\oplus y}+(\omega_1^{x\oplus y}\cdot \mathbb{Q})$ if $x \in A$ and $\A_{2i} \in \WO$ if $x \notin A$; 
    \item $\A_{2i+1}\cong \omega_1^{x\oplus y}+(\omega_1^{x\oplus y}\cdot \mathbb{Q})$.
  \end{itemize}

  Since $A \in \Sigma_1^1(y)$ and $\IF\in \Sigma^1_1$ is $\bSigma_1^1$-complete via a computable reduction, we have that for any $x\in 2^\omega$ there is an $x\oplus y$-computable tree $T_x$ such that $x \in A$ if and only if $T_x \in \IF$.
  
  By \cref{lemma:montalban1} there is a computable operator $H$ such that $H^{x\oplus y}$ is a pseudo well-order relative to $x\oplus y$. Let $T_H^{x\oplus y}$ be the tree of descending sequences in $H^{x\oplus y}$. 
  Then, for any $i$, let $\A_{2i}:=KB(T_x* T_H^{x\oplus y})\cdot \omega$ and $\A_{2i+1}:=KB(T_H^{x\oplus y})\cdot \omega$.
  Notice that $T_H^{x\oplus y} \in \IF$ regardless of whether $x \in A$. However, clearly $x\in A$, if and only if $T_x* T_H^{x\oplus y}\in \IF$ and any $z\in[T_x* T_H^{x\oplus y}]$  computes a path in $[T_H^{x\oplus y}]$ and thus $z$ cannot be hyperarithmetic in $x\oplus y$.
 
By \cref{lemma:illfoundedlinearorder} we have that
  \begin{align}
    \label{eq:xinA}x \in A &\implies T_x*T_H^{x\oplus y} \in \IF \implies \A_{2i}\cong \omega_1^{x\oplus y}+(\omega_1^{x\oplus y}\cdot \mathbb Q)\cong \A_{2i+1} \not\in \WO,\\
    \label{eq:xninA}x \notin A & \implies T_x*T_H^{x\oplus y} \notin \IF \implies \A_{2i+1}\not\cong KB(T_x*T_H^{x\oplus y})\cdot \omega=\A_{2i} \in \WO
  \end{align}
  where the isomorphism relations follow from \cref{lemma:illfoundedlinearorder,lemma:wellorderprop}. Now, if $x\in A$, then all structures in $\vec \A$ are isomorphic, so $\vec \A$ is trivially $\ex$-learnable and thus $\vec \A\in \rep$. On the other hand, if $x\not\in A$, then by \cref{lemma:wellorderprop} $\A_{2i}\equiv_2 \A_{2i+1}$ for every $i$ and thus $x\not\in \rep$ by \cref{thm:bfchar}.
\end{proof}
\begin{lemma}
\label{lem:norepborel}
  The set $\norep$ is Borel.
\end{lemma}
  \begin{proof}
  This follows immediately from the definition of and the equivalence between \cref{it:bfchar_ex} and Item 3a in \cref{thm:bfchar}.
  \end{proof}
Combining \Cref{lem:bcsigma11complete,lem:norepborel} we obtain the following.
\begin{theorem}
    \label{thm:repnorep}
    There is no Borel function from $Mod(\tau)^\nat$ to $Mod(\tau)^\nat$ such that $\vec\A \in \rep$ if and only if $\vec \A \in \norep$.
\end{theorem}

		\printbibliography
\end{document}